\documentclass[12pt]{amsart}

\usepackage{amsfonts,amsmath,amsthm}
\usepackage{latexsym}

\newtheorem{theorem}{Theorem}[section]
\newtheorem{corollary}[theorem]{Corollary}

\newtheorem{remark}[theorem]{Remark}
\newtheorem{problem}[theorem]{Problem}

\newtheorem{proposition}[theorem]{Proposition}

\theoremstyle{definition}

\newcommand{\dst}{\displaystyle}

\newcommand{\ZZ}{\ensuremath{\mathbb{Z}}}
\newcommand{\NN}{\ensuremath{\mathbb{N}}}

\newcommand{\E}{{\bf E}}
\newcommand{\D}{{\bf D}}
\newcommand{\p}{{\bf P}}

\def \e {\varepsilon}

\newcommand{\vm}{\ensuremath{\mathbf{m}}}

\newcommand{\vu}{\ensuremath{\mathbf{u}}}

\newcommand{\vx}{\ensuremath{\mathbf{x}}}
\newcommand{\vs}{\ensuremath{\mathbf{s}}}
\newcommand{\vw}{\ensuremath{\mathbf{w}}}
\newcommand{\vv}{\ensuremath{\mathbf{v}}}
\newcommand{\vk}{\ensuremath{\mathbf{k}}}
\newcommand{\vy}{\ensuremath{\mathbf{y}}}
\newcommand{\vz}{\ensuremath{\mathbf{z}}}

\def \< {\langle}
\def \> {\rangle}

\begin{document}

\title[On the existence of flat orthogonal matrices]{On the existence of flat orthogonal matrices}

\author[Ph. Jaming]{Philippe Jaming}
\address{Institut de Math\'ematiques de Bordeaux UMR 5251,
Universit\'e de Bordeaux, cours de la Lib\'eration, F 33405 Talence cedex, France}
\email{Philippe.Jaming@math.u-bordeaux1.fr}

\author[M. Matolcsi]{Mat\'e Matolcsi}
\address{M. M.: Alfr\'ed R\'enyi Institute of Mathematics,
Hungarian Academy of Sciences POB 127 H-1364 Budapest, Hungary
Tel: (+361) 483-8307, Fax: (+361) 483-8333}
\email{matolcsi.mate@renyi.mta.hu}

\thanks{M. M. was supported by OTKA 109789 and ERC-AdG 321104.
P.J. was supported by the French ANR programs ANR
2011 BS01 007 01 (GeMeCod), ANR-12-BS01-0001 (Aventures).
This study has been carried out with financial support from the French State, managed
by the French National Research Agency (ANR) in the frame of the 'Investments for
the future' Programme IdEx Bordeaux - CPU (ANR-10-IDEX-03-02).\\
This research was partially sponsored by the French-Hungarian CNRS-MTA project 135724 (HAGACAANT)}


\begin{abstract}
In this note we investigate the existence of flat orthogonal matrices, i.e. real orthogonal matrices with all entries having absolute value close to $\frac{1}{\sqrt{n}}$. Entries of $\pm \frac{1}{\sqrt{n}}$ correspond to Hadamard matrices, so the question of existence of flat orthogonal matrices can be viewed as a relaxation of the Hadamard problem.
\end{abstract}

\maketitle

\bigskip

\section{Introduction}

Let $M$ be a real orthogonal matrix of size $n\times n$. We are interested in the smallest and largest modulus among the entries of $M$:
\begin{equation}\label{lmum}
l_M:=\min_{1\le i,j\le n} |m_{i,j}|, \ \textrm{and} \ u_M:=\max_{1\le i,j\le n} |m_{i,j}|
\end{equation}
and, more precisely, in estimating the maximal possible value for $l_M$ and the minimal possible value of $u_M$.
In other words, we want to estimate the following quantities:
\begin{equation}\label{lnun}
l_n:=\max_{M\in O(n)}l_M  \quad\textrm{and} \quad
u_n:=\min_{M\in O(n)}u_M,
\end{equation}
where $O(n)$ denotes the group of orthogonal matrices.

\medskip

In order to determine $l_n$ (respectively, $u_n$) one has to control the lowest (resp. uppermost)
absolute value in an orthogonal matrix $M$. It is also natural to ask whether one can control both
quantities simultaneously. For this purpose we introduce the measure of
``flatness'' of an $n\times n$ orthogonal matrix $M$ as
\begin{equation}\label{fm}
f_M=\min \{\e>0 : \frac{1-\e}{\sqrt{n}}\le |m_{i,j}|\le  \frac{1+\e}{\sqrt{n}}, \ \ \textrm{for all} \ 1\le i,  j\le n \},
\end{equation}
and we are interested in how flat an orthogonal matrix can be:
\begin{equation}\label{fn}
f_n:=\min_{M\in O(n)}f_M.
\end{equation}

It is trivial to see that $l_n\le \frac{1}{\sqrt{n}}$, $u_n\ge \frac{1}{\sqrt{n}}$ for all $n$, and the famous Hadamard conjecture states that if $4|n$ then $l_n=u_n=\frac{1}{\sqrt{n}}$. It is also easy to see that
$f_n\geq \frac{c}{n}$, unless there exists an $n\times n$ Hadamard matrix (Proposition \ref{smallerror} below).
This note was motivated by the natural question of L. Baratchart (personal communication), as to whether $l_n\ge \frac{c}{\sqrt{n}}$ for some absolute constant $c$. The answer to this question is positive and provided by Theorem \ref{thm:block} below. We will also prove $u_n\le \frac{c}{\sqrt{n}}$ with another absolute constant $c$ in Theorem \ref{thm:un}. However, we list here some more restrictive questions on flat orthogonal matrices, all of which we have only partial answers to.

\begin{problem}\label{probln}
Is it true that $l_n=(1+o(1))\frac{1}{\sqrt{n}}$?
\end{problem}

\begin{problem}\label{probun}
Is it true that $u_n=(1+o(1))\frac{1}{\sqrt{n}}$?
\end{problem}
\noindent Can we control $l_n$ and $u_n$ simultaneously?
\begin{problem}\label{probfn}
Is it true that $f_n=o(1)$?
\end{problem}

These questions can be seen as relaxations of the Hadamard problem,
and the matrices leading to such bounds could be called almost Hadamard matrices.
However, a different notion of {\em almost Hadamard matrices} was already introduced and considered in \cite{banica1, banica2} (in those papers the emphasis is on various matrix norms and not on the entries of the matrix). Therefore, to make a clear distinction, we prefer to use the terminology {\it flat orthogonal matrices} here.

Finally, circulant matrices play a special role in applications, therefore we can add this as an additional constraint.

\begin{problem}\label{probcircfn}
Let $f_n^{circ}:=\min_{M\in O(n)\cap Circ(n)}f_M$, where $Circ(n)$ denotes the set of $n\times n$ circulant matrices and $f_M$ was defined in \eqref{fm}. Is it true that $f_n^{circ}=o(1)$?
\end{problem}

We firmly believe that the answer to the first three problems is positive, while we are undecided as to the last one. Recall that a conjecture of Ryser asserts that there are no $n\times n$ circulant Hadamard matrices if $n>4$.

\medskip

The remaining of this paper is split into two sections. The first one is devoted to general constructions that lead to bounds on $u_n$ and $l_n$ valid in arbitrary dimension. We then devote the last section to improved bounds when the size of the matrix have various arithmetic properties.

\section{General constructions}

We begin by a simple result which shows that the flatness parameter $f_n$ cannot be expected to be very small. In other words, requiring $f_n$ to be very small is equivalent to requiring a Hadamard matrix of order $n$ to exist.

\begin{proposition}\label{smallerror}
Let $n\geq 3$ and let $\varepsilon>0$ be such that
$$
\varepsilon <\begin{cases}
                 1/n&\mbox{if $n$ is odd}\\
                 2/n&\mbox{if $n$ is even}
                \end{cases}.
$$
 Assume there exists an orthogonal matrix $M$ such that
 for every $j,k=1,\dots, n$,
 \begin{equation}
  \label{eq:almosthadamard}
  \left(\frac{1-\varepsilon}{n}\right)^{1/2}\leq |m_{j,k}|\leq \left(\frac{1+\varepsilon}{n}\right)^{1/2}.
 \end{equation}
Then $n$ is a multiple of $4$ and there exists a Hadamard matrix of order $n$.
\end{proposition}
\begin{proof}
Let $S$ denote the matrix defined by the sign of the entries of $M$, i.e. $s_{j,k}=\frac{1}{\sqrt{n}}\textrm{sign}\,m_{j,k}$. Consider two rows $\vs_j$ and $\vs_k$ of $S$. Let $r$ denote the number of columns where the entries in $\vs_j$ and $\vs_k$ match, and $n-r$ where they differ. Then, for the corresponding rows $\vm_j, \vm_k$ of $M$ we have
\begin{eqnarray*}
0&=&\langle \vm_j,\vm_k\rangle\\
&&\begin{cases}\dst\leq \frac{r}{n}(1+\varepsilon)- \frac{n-r}{n}(1-\varepsilon)
   =\frac{2r-n+\varepsilon n}{n}\\
  \dst \geq \frac{r}{n}(1-\varepsilon)- \frac{n-r}{n}(1+\varepsilon)
   =\frac{2r-n-\varepsilon n}{n}
  \end{cases}.
\end{eqnarray*}
However, if $n$ is odd the interval  $(2r-n-\varepsilon n, 2r-n+\varepsilon n)$ does not contain zero (or any even number), a contradiction. If $n$ is even, the interval  $(2r-n-\varepsilon n, 2r-n+\varepsilon n)$ contains zero if and only if $r=\frac{n}{2}$, in which case the corresponding rows $\vs_j, \vs_k$ are also orthogonal, and we conclude that $S$ is a Hadamard matrix.
\end{proof}

We continue with a simple block construction which proves that $l_n$ is at least as large as $\frac{1}{2\sqrt{n}}$.

\begin{theorem}\label{thm:block}
For any dimension $n$ we have $l_n\ge \frac{1}{2\sqrt{n}}$.
\end{theorem}
\begin{proof}
Assume $n=2^r+q$ where $q<2^r$, and introduce the notation $s=2^r-q$ for brevity. Let $H$ be a Hadamard matrix of order $2^r$ (of course, such a matrix exists, e.g. the tensorial power $F_2^{\otimes r}$). We use the normalization that $H$ has entries $\pm\frac{1}{\sqrt{2^r}}$ instead of $\pm 1$. Let $\tilde{M}$ be the extension of $H$ by an identity matrix of order $q$ in the lower right. Split $\tilde{M}$ into blocks of size $s$ and $q$ and $q$ as follows (the indexes simply indicating the sizes of the blocks):
$$
\tilde{M}=\left[
\begin{array}{c|c|c}
H_{s,s} & H_{s,q} & 0\\ \hline
H_{q,s} & H_{q,q} & 0\\ \hline
0 & 0 & I_q
\end{array}\right].
$$

Let $U$ denote the following orthogonal block-matrix:
$$
U=\left[
\begin{array}{c|c|c}
I_{s} & 0 & 0\\ \hline
0 & \frac{1}{\sqrt{2}}I_{q} & -\frac{1}{\sqrt{2}}I_{q}\\ \hline
0 & \frac{1}{\sqrt{2}}I_{q} & \frac{1}{\sqrt{2}}I_{q}
\end{array}\right],
$$
and let $M:=U^T \tilde{M} U$. A direct calculation shows that
$$
M=\left[
\begin{array}{c|c|c}
H_{s,s} & \frac{1}{\sqrt{2}}H_{s,q} & -\frac{1}{\sqrt{2}}H_{s,q}\\ \hline
\frac{1}{\sqrt{2}}H_{q,s} & \frac{1}{2}(H_{q,q}+I_q) & \frac{1}{2}(-H_{q,q}+I_q)\\ \hline
-\frac{1}{\sqrt{2}}H_{q,s} & \frac{1}{2}(-H_{q,q}+I_q) & \frac{1}{2}(H_{q,q}+I_q)
\end{array}\right].
$$
The smallest modulus among the entries of $M$ is $\frac{1}{2\sqrt{2^r}}\ge \frac{1}{2\sqrt{n}}$, which proves the theorem. (Note, however, that the largest appearing modulus is $\frac{1}{2}+\frac{1}{\sqrt{2^r}}$, so this construction gives no indication with respect to Problem \ref{probfn}.)
\end{proof}

Proving that $u_n\le \frac{c}{\sqrt{n}}$ is also fairly easy, as one can use a block-diagonal construction.

\begin{theorem}\label{thm:un}
For any dimension $n$ we have $u_n\le \frac{2+o(1)}{\sqrt{n}}$.
\end{theorem}
\begin{proof}
If $A_1\in O(n_1)$ and $A_2\in O(n_2)$ are orthogonal matrices then the block-diagonal matrix $A_1\oplus A_2$ is an orthogonal matrix of order $n_1+n_2$. This implies $u_{n+m}\le \max\{u_n, u_m\}$. Also, Propositions \ref{had4km1} and \ref{n4kp1} below show that $u_p=\frac{1+o(1)}{\sqrt{p}}$ whenever the dimension $p$ is a prime. For general $n$ we must invoke the following weak-Goldbach type result from number theory: for every $\e>0$, every large enough odd number $n$ can be written as a sum of three primes $n=p_1+p_2+p_3$ where each $p_i$ lies in the interval $[(1-\e)n/3, (1+\e)n/3]$, while every large enough even number $n$ can be written as a sum of four primes $n=p_1+p_2+p_3+p_4$ where each $p_i$ lies in the interval $[(1-\e)n/4, (1+\e)n/4]$. This follows from the method of Vinogradov \cite{vin}. Hence, the block-diagonal construction together with this weak-Goldbach type result implies $u_n\le \frac{\sqrt{3}+o(1)}{\sqrt{n}}$ if $n$ is odd, and $u_n\le \frac{2+o(1)}{\sqrt{n}}$ if $n$ is even.
\end{proof}

\begin{remark}\rm
In connection with $l_n$ and $u_n$ it is natural to examine how the entries of a typical random orthogonal matrix behave.
It is well known that for any coordinate $z$ of a random unit vector we have $Pr(|z|>\frac{t}{\sqrt{n}})\le e^{-t^2/2}$
(this can be seen by upper bounding the area of a spherical cap of radius $r$ by that of a sphere of radius $r$).
This implies, via a simple calculation, that $u_M\le \frac{c\sqrt{\log n}}{\sqrt{n}}$ for a random orthogonal matrix $M$,
with high probability. Therefore, a random orthogonal matrix is typically ``not far'' from the bound $u_M\le \frac{c}{\sqrt{n}}$
given in Theorem \ref{thm:un}. On the contrary, the lowest absolute value in a random unit vector is smaller than $\frac{c}{n}$
with high probability (this can be seen by generating a random unit vector as the normalized vector of $n$ independent
Gaussian random variables). The same holds, a fortiori, for a random orthogonal matrix, which shows that the lowest
entry is typically very far from the optimal bound given in Theorem \ref{thm:block}.
\hfill $\square$
\end{remark}

\begin{remark}\rm
If we assume that the Hadamard conjecture holds then $l_n=u_n=\frac{1}{\sqrt{n}}$ for all $n$ divisible by 4. By the simple construction of Proposition \ref{had4km1} we also conclude that for $n\equiv 3 \ (mod \ 4)$ the quantities $l_n$ and $u_n$ are both of the magnitude $(1+o(1))\frac{1}{\sqrt{n}}$. However, we could not prove such a statement for the case $n\equiv 1, \ 2 \ (mod \ 4)$.
\end{remark}

\section{Specific constructions}

In the rest of this note we give some positive partial results with respect to Problems \ref{probln}, \ref{probfn}, \ref{probcircfn}. As noted in \cite[Section 3]{banica1} symmetric balanced incomplete block designs give rise to orthogonal matrices with two entries. If the parameters of the block design are suitable then the entries will be close to $\pm \frac{1}{\sqrt{n}}$. Namely, we have the following special case.

\begin{proposition}\label{had4km1}
(i) If the dimension $n$ is such that a Hadamard matrix $H$ of size $(n+1)\times (n+1)$ exists, then there exists an orthogonal matrix $M$ of size $n\times n$ with all entries having modulus $(1+o(1))\frac{1}{\sqrt{n}}$. In particular, this is the case if $n=p^r$ where $p=4k-1$ is a prime and $r$ is odd.

(ii) If the dimension $n$ is a prime of the form $4k-1$, then $M$ can be chosen to be circulant.
\end{proposition}
\begin{proof}
(i) We can assume that the Hadamard matrix $H$ is in standard form, i.e. the first row and column of $H$ consist of 1's. Delete the first row and column of $H$, and in the remaining matrix $H'$ replace the entries $+1$ and $-1$ by the variables $x$ and $y$, respectively. Orthogonality of any two rows of $H'$ is now equivalent to $\frac{n-3}{4}x^2+\frac{n+1}{4}y^2+ \frac{n+1}{2}xy=0$, and the unit length of the rows of $H'$ is ensured by $\frac{n-1}{2}x^2+\frac{n+1}{2}y^2=1$. This system of equations admits the (non-unique) solution
$$x=\frac{-1}{\sqrt{n+1}-1}$$
$$y=\frac{2}{\sqrt{n+1}}-\frac{1}{\sqrt{n+1}-1},$$
and we define $M$ as the $n\times n$ matrix with these values. Note that both $x$ and $y$ have the order of magnitude $(1+o(1))\frac{1}{\sqrt{n}}$. In fact, the error term $o(1)$ here has the order of magnitude $O(\frac{1}{\sqrt{n}})$. We remark that $M$ corresponds to the {\it Hadamard design} associated to $H$, and this construction is a special case of the one described in \cite[Section 3]{banica1}. If $n=p^r$ where $p=4k-1$ is a prime and $r$ is odd, then a Hadamard matrix of size $(n+1)\times (n+1)$ exists by the Paley construction.

(ii) If $n$ is a prime of the form $4k-1$ then the Paley construction essentially leads to a circulant matrix. Namely, consider the following circulant matrix $M$ corresponding to the quadratic character of $\mathbb{F}_p$: $[M]_{i,j}=x$ if $i-j$ is a quadratic residue, and $[M]_{i,j}=y$ if $i-j$ is a non-residue or zero. The above values of $x,y$ ensure orthogonality of $M$.
\end{proof}

When $n$ is a prime of the form $4k+1$, the construction is much less trivial, as described below.

\begin{proposition}\label{n4kp1}
If the dimension is a prime $p=4k+1$, then there exists a circulant orthogonal matrix $M$ of size $p\times p$ with all entries having modulus $(1+o(1))\frac{1}{\sqrt{p}}$.
\end{proposition}
\begin{proof}
For this proof it will be convenient to identify the cyclic group $\ZZ_p$ with the numbers $(\frac{-p+1}{2}, \dots, \frac{p-1}{2})$, and use coordinates of vectors accordingly. Also, let $Q, NQ\subset \ZZ_p$ denote the set of quadratic residues and non-residues, respectively (0 is not included in either $Q$ or $NQ$). Consider the vector $\vv=(v_\frac{-p+1}{2}, \dots, v_\frac{p+1}{2})$ given by the quadratic character, i.e. $v_0=0$, and for nonzero $j$ we have $v_j=\pm 1$ according to whether $j\in Q$ or $j\in NQ$. Note that $\vv$ is symmetric, $v_j=v_{-j}$, because $p=4k+1$. Note also that $\hat \vv(k)=\sum_{j} v_j e^{2\pi ijk/p}=\pm \sqrt{p}$ for $k\ne 0$, and $\hat \vv(0)=0$. We will need a random modification of $\vv$. The construction is analogous to the one given in \cite[Theorem 9.2]{ruzsmat}.

\medskip

Fix $\e>0$, and let $\rho=\frac{1}{\sqrt{p}}$. As $NQ$ is symmetric, it can be written as a disjoint union of a set $H$ and its negative, $NQ=H\cup -H$, in other words $NQ=\cup_{y\in H} \{y, -y\}$. For each $y\in H$ consider independent random variables $\xi_y$ such that $\p(\xi_y=1/2)=\rho$, $\p(\xi_y=0)=1-\rho$. For $y=0$ set $\p(\xi_y=1/4)=\rho$, $\p(\xi_y=0)=1-\rho$. Consider the random vector $\vw=(w_\frac{-p+1}{2}, \dots, w_\frac{p-1}{2})$ given by $w_y=\xi_y$ if $y\in H$ or $y=0$, $w_y=\xi_{-y}$ if $-y\in H$, and $w_y=0$ if $y\in Q$. Let us evaluate the Fourier transform of the random vector $\vw$.

\begin{equation}
\E(\hat \vw(k))=\frac{1}{2}\rho \left(\frac{1}{2}+\sum_{y\in NQ}e^{2\pi i yk/p}\right)= \begin{cases}
    \sqrt{p}/4 & \text{if } k=0, \\
    \pm 1/4 & \text{if } k\ne 0
  \end{cases}
\end{equation}
and, for all $k$,
\begin{equation}
\D^2(\hat \vw(k))=\frac{1}{4}\rho(1-\rho) \left(\frac{1}{4}+\sum_{y\in H}(e^{2\pi i yk/p}+e^{-2\pi i yk/p})^2\right)\le
    \frac{1}{4}\sqrt{p}.
\end{equation}

We invoke here a large deviation inequality of Chernov as stated in \cite[Theorem 1.8]{taovu06}: let $X_1, \ldots, X_n$ be independent random variables satisfying $| X_i - \E(X_i)| \leq 1 $ for all $i$. Put
  $ X = X_1 + \ldots + X_n$ and let $\sigma^2$ be the variance of $X$. For any $t>0$ we have
   \[ \p( |X - \E(X)| \geq t\sigma) \leq 2 \max \left( e^{-t^2/4}, e^{-t\sigma/2} \right) .\]

\medskip

Using this estimate with $t=p^\e$ we obtain

\begin{equation}
\p\left(|\hat \vw(0)-\sqrt{p}/4|\ge \frac{1}{2}p^{\frac{1}{4}+\e}\right)\le 2e^{-\frac{p^{2\e}}{4}},
\end{equation}

\begin{equation}
\p\left(|\hat \vw(k) \mp \frac{1}{4}|\ge \frac{1}{2}p^{\frac{1}{4}+\e}\right)\le 2e^{-\frac{p^{2\e}}{4}} \ \text{for all } k\ne 0.
\end{equation}

Therefore, with high probability none of the above events occur, and we have $|\hat \vw(0)-\sqrt{p}/4|\le O(p^{\frac{1}{4}+\e})$ and $|\hat \vw(k)|\le O(p^{\frac{1}{4}+\e})$. Fix such a favourable vector $\vw$. We can also assume without loss of generality that $\xi_0=1/4$ and hence $w_0=1/4$ (we are free to change $w_0$ from 0 to $1/4$, if necessary, without altering the order of magnitude of $\hat \vw$).

\medskip

Finally, consider the vector $\vz=\vv+4\vw$, and let $\vu=\frac{1}{p}\hat{\vz}$. The vector $\vz$ is unimodular (some of the $-1$ entries in $\vv$ were changed to $+1$ and the value at 0 was changed to $+1$). Therefore the circulant matrix $M$ with first row $\vu$ is orthogonal. Also, the entries of $\vu$ are all of absolute value $\frac{1}{\sqrt{p}}(1+p^{-\frac{1}{4}+\e})$, by construction.

\medskip

We remark that a similar construction works if the dimension $p$ is of the form $4k-1$ but the result is inferior to Proposition \ref{had4km1} in the sense that the error term is larger.
\end{proof}

By combining the results of the propositions above we can answer Problem \ref{probfn} and \ref{probcircfn} for dimensions $n$ which are composed of large prime factors.

\begin{corollary}\label{cor1}
For any fixed $m$ let $\NN_m$ denote the set of positive integers $n=2^s p_1 \dots p_r$ such that each odd prime factor $p_j\ge n^{1/m}$ (the primes may appear with multiplicity). \\
(i) For every $m$ and every $n\in \NN_m$ there exists an orthogonal matrix $M$ of size $n\times n$ with all entries having modulus $(1+o_m(1))\frac{1}{\sqrt{n}}$. \\
(ii) If all the odd primes $p_j$ appearing in the factorization of $n$ are distinct and $s=0$ or $s=2$ then there exists a circulant orthogonal matrix $M$ of size $n\times n$ with all entries having modulus $(1+o_m(1))\frac{1}{\sqrt{n}}$.
\end{corollary}
\begin{proof}
(i) Let $n=2^sp_1 \dots p_r\in \NN_m $. Then $r\le m$ by the definition of $\NN_m$. Let $H$ denote a Hadamard matrix of size $2^s$ (such matrix exists, of course, e.g. $H=F_2^{\otimes s}$), and let $M_{p_j}$ denote the matrices corresponding to the primes $p_j$, as constructed in Propositions \ref{had4km1} and \ref{n4kp1}. Let $M$ be the tensorial product of all the matrices $H$ and $M_{p_j}$ (where $M_{p_j}$ is taken with the same multiplicity as $p_j$ in $n$). As $r\le m$ for any $n\in \NN_m$, the errors do not accumulate, and we still have $|[M]_{i,j}|=(1+o_m(1))\frac{1}{\sqrt{n}}$.

\medskip

(ii) The matrices $M_{p_j}$ are circulant by the constructions of Propositions \ref{had4km1} and \ref{n4kp1}. Let $\vx_j$ denote the first row of $M_{p_j}$.  If $s=0$, $n=p_1 \dots p_r$ with all $p_j$ distinct, then $\ZZ_n\equiv \ZZ_{p_1}\dots \ZZ_{p_r}$. Let $\vk=(k_1, \dots, k_r)\in \ZZ_{p_1}\dots \ZZ_{p_r}$ and let $\vy(\vk)=\prod_{j=1}^r \vx_j(k_j)$. Then $\vy$ generates a circulant matrix which is orthogonal (because all $\hat \vx_j$ are unimodular, and hence so is $\hat \vy$), and the entries of $\vy$ are of absolute value $(1+o_m(1))\frac{1}{\sqrt{n}}$. When $s=2$ we can incorporate the $4\times 4$ real circulant Hadamard matrix in the same manner.
\end{proof}

\begin{remark}\rm
Problem \ref{probcircfn} concerning the circulant case has an interesting connection to {\it ultraflat polynomials}. It is well-known that $\vx=(x_1, \dots, x_n)$ generates a circulant orthogonal matrix if and only if the Fourier transform $\hat \vx=(w_1, \dots, w_n)$ is unimodular on $\hat \ZZ_n$, i.e. $|w_j|=|w_k|$ for all $j, k$. If $\vx$ is real, then $\vw$ is conjugate symmetric, i.e. $w_j=\overline w_{n-j}$. We also want that all $|x_j|\approx \frac{1}{\sqrt{n}}$. Considering $w_1, \dots, w_n$ as variables we are led to the problem of constructing a polynomial $P(z)=\sum_{j=1}^n w_jz^j$ where $w_j=\overline w_{n-j}$ and $|w_j|=1$ (after re-normalization), such that $P(z)$ is ``flat'' at the $n$th roots of unity, i.e. $|P(\omega^j)|=(1+o(1))\sqrt{n}$, where $\omega=e^{2i\pi/n}$ and $j=0, \dots n-1$. Dropping the restriction $w_j=\overline w_{n-j}$ one can require $P(z)$ to be flat all over the unit circle (i.e. $|P(z)|=(1+o(1))\sqrt{n}$ for all $|z|=1$), and such polynomials are called {\it ultraflat}. The existence of ultraflat polynomials was proven by Kahane \cite{kahane}. However, the extra condition $w_j=\overline w_{n-j}$ prevents $P(z)$ from being ultraflat as shown by Remark 5.1 in \cite{erdelyi}: the restriction $w_j=\overline w_{n-j}$ implies $\max_{|z|=1} |P(z)|\ge (1+\e)\sqrt{n}$ with $\e=\sqrt{4/3}-1$. However, this does not answer Problem \ref{probcircfn} because we require $P(z)$ to be flat only at the $n$th roots of unity. Problem \ref{probcircfn} can therefore be regarded as a discretized version (at the $n$th roots of unity) of the question of existence of ultraflat polynomials with the self-conjugacy restriction $w_j=\overline w_{n-j}$. 

Notice also that Corollary \ref{cor1} gives an affirmative answer to Problem \ref{probcircfn} for some dimensions $n$ (e.g. when $n=4p$, $p$ being a prime). Therefore, trying to prove Ryser's conjecture on the non-existence of circulant Hadamard matrices by giving a negative answer to Problem \ref{probcircfn} cannot possibly work. In the other direction, hoping to construct a circulant Hadamard matrix by first constructing a circulant flat orthogonal matrix and then modifying its entries is also rather naive and hopeless in our opinion. \hfill $\square$
\end{remark}

We saw in Proposition \ref{had4km1} that given any Hadamard matrix $H$ one can reduce the dimension by 1, and construct a flat orthogonal matrix. It is natural to try also to increase the dimension by 1.
The general construction given in Proposition \ref{smallerror} allows to do so, but will introduce an entry of size $1/2$, thus destroying
the flatness. Our last specific construction concerns the increase of dimension by $1$ without destroying flatness, but it only works under some restrition on the dimension. We recall that a Hadamard matrix is called {\it regular} if the row sums and column sums of $H$ are all equal.

\begin{proposition}\label{regular}
Assume the dimension $n$ is such that a regular Hadamard matrix $H$ of order $n$ exists (this implies $n=4k^2$ for some $k$). Then there exists an orthogonal matrix $M$ of size $(n+1)\times (n+1)$ with all entries having modulus $(1+o(1))\frac{1}{\sqrt{n+1}}$.
\end{proposition}

\begin{proof}
Let $b=\frac{1-(n+1)^{-1/2}}{n}$ and replace the positive entries of $H$ by $\frac{1}{\sqrt{n}}-b$, while the negative entries with $-\frac{1}{\sqrt{n}}-b$ (the sign of $b$ is purposefully negative in both cases). Next, extend this matrix by a new row and column filled with entries $a=\frac{-1}{\sqrt{n+1}}$, and let $M$ be the arising matrix. The entries of $M$ are of modulus $(1+o(1))\frac{1}{\sqrt{n+1}}$ (in fact, the error term is of the order $\frac{1}{\sqrt{n}}$), and an easy calculation shows that $M$ is orthogonal.
\end{proof}

\medskip

We end this note by emphasizing that Problems \ref{probln}, \ref{probfn}, \ref{probcircfn} all remain open for dimensions $n$ with small prime factors, i.e. for dimensions $n$ which are not covered by Corollary \ref{cor1}.


\begin{thebibliography}{11}

\bibitem{ortn1}
\textsc{T. Banica, B. Collins, J-M. Schlenker},
\newblock{\em On orthogonal matrices maximizing the 1-norm},
Indiana Univ. Math. J. {\bf 59}, 839--856 (2010).

\bibitem{banica2}
\textsc{T. Banica, I. Nechita},
\newblock{\em Almost Hadamard matrices: the case of arbitrary exponents},
Discrete Appl. Math., {\bf 161}, 2367--2379 (2013).

\bibitem{banica1}
\textsc{T. Banica, I. Nechita and  K.~{\.Z}yczkowski},
\newblock{\em Almost Hadamard matrices: general theory and examples},
Open Syst. Inf. Dyn. {\bf 19}, 1--26 (2012).

\bibitem{erdelyi}
\textsc{T. Erd\'elyi},
\newblock{\em The phase problem of ultraflat unimodular polynomials: The resolution
of the conjecture of Saffari},
Math. Ann. {\bf 321}, 905--924 (2001).

\bibitem{kahane}
\textsc{J.P. Kahane},
\newblock{\em Sur les polynomes a coefficient unimodulaires},
Bull. London Math. Soc. {\bf 12}, 321--342, (1980).

\bibitem{ruzsmat}
\textsc{M. Matolcsi, I. Z. Ruzsa},
\newblock{\em Difference sets and positive exponential sums I. General properties},
Journal of Fourier Anal. Appl., {\bf 20}, 17--41, (2014).

\bibitem{taovu06}
T.~Tao and V.~H. Vu, \emph{Additive combinatorics}, Cambridge University Press,
  Cambridge, 2006.

\bibitem{vin}
\textsc{M. Vinogradov},
\newblock{\em A new method in analytic number theory (Russian)},
Tr. Mat. Inst. Steklova, {\bf 10}, 5--122, (1937).

\end{thebibliography}
\end{document}